\documentclass[a4paper,twoside]{article}
\usepackage{a4}
\usepackage{amssymb}
\usepackage{amsmath}
\usepackage{upref}

\usepackage[active]{srcltx}
\usepackage[dvipsnames]{color}
\usepackage[pagebackref,colorlinks,citecolor=blue,linkcolor=blue,urlcolor=blue]{hyperref}
\allowdisplaybreaks[2] % To make it possible for displaybreaks

\newcommand{\prtid}{}
\makeatletter
\def\sectionmark#1{} %\markboth{{\sectnr #1}}{{\sectnr #1}}} %Journal
\def\subsectionmark#1{}
\newcommand{\sectnr}{\ifnum \c@secnumdepth >\z@
                 \thesection.\hskip 1em\relax \fi}
\def\@evenhead{\footnotesize\rm\thepage\hfil\leftmark\hfil\llap{\prtid}}
\def\@oddhead{\footnotesize\rm\rlap{\prtid}\hfil\rightmark\hfil\thepage}
\def\tableofcontents{\section*{Contents} %\@mkboth{Contents}{Contents}} %Journal
 \@starttoc{toc}}
\makeatother

\makeatletter
\def\@biblabel#1{#1.}
\makeatother
\makeatletter
\let\Thebibliography=\thebibliography
\renewcommand{\thebibliography}[1]{\def\@mkboth##1##2{}\Thebibliography{#1}
\addcontentsline{toc}{section}{References}
\frenchspacing % Maybe not needed
\setlength{\@topsep}{0pt}% Delete if extra space before list
\setlength{\itemsep}{0pt}%
\setlength{\parskip}{0pt plus 2pt}%
}
\makeatother
\makeatletter
\def\mdots@{\mathinner.\nonscript\!.%
 \ifx\next,.\else\ifx\next;.\else\ifx\next..\else
 \nonscript\!\mathinner.\fi\fi\fi}
\let\ldots\mdots@
\let\cdots\mdots@
\let\dotso\mdots@
\let\dotsb\mdots@
\let\dotsm\mdots@
\let\dotsc\mdots@
\def\vdots{\vbox{\baselineskip2.8\p@ \lineskiplimit\z@
    \kern6\p@\hbox{.}\hbox{.}\hbox{.}\kern3\p@}}
\def\ddots{\mathinner{\mkern1mu\raise8.6\p@\vbox{\kern7\p@\hbox{.}}%
    \raise5.8\p@\hbox{.}\raise3\p@\hbox{.}\mkern1mu}}
\makeatother
\makeatletter
\def\@seccntformat#1{\csname the#1\endcsname.\quad}
\makeatother
\makeatletter
\long\def\@makecaption#1#2{%
  \vskip\abovecaptionskip
  \sbox\@tempboxa{ #1. #2}%
  \ifdim \wd\@tempboxa >\hsize
    #1. #2\par
  \else
    \global \@minipagefalse
    \hb@xt@\hsize{\hfil\box\@tempboxa\hfil}%
  \fi
  \vskip\belowcaptionskip}
\makeatother
\makeatletter
\renewcommand\section{\@startsection {section}{1}{\z@}%
                                   {-3.5ex \@plus -1ex \@minus -.2ex}%
                                   {2.3ex \@plus.2ex}%
                                   {\normalfont\Large\bfseries\boldmath}}
\renewcommand\subsection{\@startsection{subsection}{2}{\z@}%
                                     {-3.25ex\@plus -1ex \@minus -.2ex}%
                                     {1.5ex \@plus .2ex}%
                                     {\normalfont\large\bfseries\boldmath}}
\renewcommand\subsubsection{\@startsection{subsubsection}{3}{\z@}%
                                     {-3.25ex\@plus -1ex \@minus -.2ex}%
                                     {1.5ex \@plus .2ex}%
                                     {\normalfont\normalsize\bfseries\boldmath}}
\renewcommand\paragraph{\@startsection{paragraph}{4}{\z@}%
                                    {3.25ex \@plus1ex \@minus.2ex}%
                                    {-1em}%
                                    {\normalfont\normalsize\bfseries\boldmath}}
\renewcommand\subparagraph{\@startsection{subparagraph}{5}{\parindent}%
                                       {3.25ex \@plus1ex \@minus .2ex}%
                                       {-1em}%
                                      {\normalfont\normalsize\bfseries\boldmath}}
\makeatother
\newcommand{\authortitle}[2]{\author{#1}\title{#2}\markboth{#1}{#2}}
\newcommand{\auth}[2]{{#1, #2.}}

\newcommand{\art}[6]{{\sc #1, \rm #2, \it #3\/ \bf #4 \rm (#5), \mbox{#6}.}}
\newcommand{\book}[3]{{\sc #1, \it #2, \rm #3.}}
\newcommand{\AND}{{\rm and }}

\newcommand{\artprep}[3]{{\sc #1, \rm #2, \rm #3.}}

\RequirePackage{amsthm}
\newtheoremstyle{descriptive}%
  {\topsep}   %{\medskipamount}          % Space above
  {\topsep}   %  {\medskipamount}          % Space below
  {\rmfamily} % Body font
  {}          % Indent
  {\bfseries} % Head font
  {.}         % Punctuation after thm head
  { }         % Space after thm head
  {}          % Thm head spec(?)
\newtheoremstyle{propositional}%
  {\topsep}   %  {\medskipamount}          % Space above
  {\topsep}   %  {\medskipamount}          % Space below
  {\itshape}  % Body font
  {}          % Indent
  {\bfseries} % Head font
  {.}         % Punctuation after thm head
  { }         % Space after thm head
  {}          % Thm head spec(?)
\newtheoremstyle{remarkstyle}%
  {\topsep}   %  {\medskipamount}          % Space above
  {\topsep}   %  {\medskipamount}          % Space below
  {\rmfamily}  % Body font
  {}          % Indent
  {\itshape} % Head font
  {.}         % Punctuation after thm head
  { }         % Space after thm head
  {}          % Thm head spec(?)
\theoremstyle{propositional}
\newtheorem{thm}{Theorem}[section]
\newtheorem{lem}[thm]{Lemma}
\newtheorem{prop}[thm]{Proposition}

\theoremstyle{descriptive}
\newtheorem{defi}[thm]{Definition}
\newtheorem{remark}[thm]{Remark}
\newtheorem{example}[thm]{Example}
\makeatletter
\renewenvironment{proof}[1][\proofname]{\par
  \pushQED{\qed}%
  \normalfont
  \trivlist
  \item[\hskip\labelsep
        \itshape
    #1\@addpunct{.}]\ignorespaces
}{%
  \popQED\endtrivlist\@endpefalse
}
\makeatother
\newcommand{\setm}{\setminus}
{\catcode`p =12 \catcode`t =12 \gdef\eeaa#1pt{#1}}      % Get slantfactor
\def\accentadjtext#1{\setbox0\hbox{$#1$}\kern   % Convert it with height
                \expandafter\eeaa\the\fontdimen1\textfont1 \ht0 }
\def\accentadjscript#1{\setbox0\hbox{$#1$}\kern % Convert it with height
                \expandafter\eeaa\the\fontdimen1\scriptfont1 \ht0 }
\def\accentadjscriptscript#1{\setbox0\hbox{$#1$}\kern   % Convert it with height
                \expandafter\eeaa\the\fontdimen1\scriptscriptfont1 \ht0 }
\def\accentadjtextback#1{\setbox0\hbox{$#1$}\kern       % Convert it with height
                -\expandafter\eeaa\the\fontdimen1\textfont1 \ht0 }
\def\accentadjscriptback#1{\setbox0\hbox{$#1$}\kern     % Convert it with height
                -\expandafter\eeaa\the\fontdimen1\scriptfont1 \ht0 }
\def\accentadjscriptscriptback#1{\setbox0\hbox{$#1$}\kern % Convert it with height
               -\expandafter\eeaa\the\fontdimen1\scriptscriptfont1 \ht0 }
\def\itoverline#1{{\mathsurround0pt\mathchoice
        {\rlap{$\accentadjtext{\displaystyle #1}
                \accentadjtext{\vrule height1.593pt}
                \overline{\phantom{\displaystyle #1}
                \accentadjtextback{\displaystyle #1}}$}{#1}}
        {\rlap{$\accentadjtext{\textstyle #1}
                \accentadjtext{\vrule height1.593pt}
                \overline{\phantom{\textstyle #1}
                \accentadjtextback{\textstyle #1}}$}{#1}}
        {\rlap{$\accentadjscript{\scriptstyle #1}
                \accentadjscript{\vrule height1.593pt}
                \overline{\phantom{\scriptstyle #1}
                \accentadjscriptback{\scriptstyle #1}}$}{#1}}
        {\rlap{$\accentadjscriptscript{\scriptscriptstyle #1}
                \accentadjscriptscript{\vrule height1.593pt}
                \overline{\phantom{\scriptscriptstyle #1}
                \accentadjscriptscriptback{\scriptscriptstyle #1}}$}{#1}}}}
                          
\def\vint{\mathop{\mathchoice%
          {\setbox0\hbox{$\displaystyle\intop$}\kern 0.22\wd0%
           \vcenter{\hrule width 0.6\wd0}\kern -0.82\wd0}%
          {\setbox0\hbox{$\textstyle\intop$}\kern 0.2\wd0%
           \vcenter{\hrule width 0.6\wd0}\kern -0.8\wd0}%
          {\setbox0\hbox{$\scriptstyle\intop$}\kern 0.2\wd0%
           \vcenter{\hrule width 0.6\wd0}\kern -0.8\wd0}%
          {\setbox0\hbox{$\scriptscriptstyle\intop$}\kern 0.2\wd0%
           \vcenter{\hrule width 0.6\wd0}\kern -0.8\wd0}}%
          \mathopen{}\int}

\newcommand{\grad}{\nabla}
\newcommand{\disp}{\displaystyle}
\DeclareMathOperator{\Div}{div}
\DeclareMathOperator{\cp}{cap}
\newcommand{\bdry}{\partial}
\newcommand{\bdy}{\bdry}
\newcommand{\loc}{_{\rm loc}}
\newcommand{\noi}{\noindent}
\newcommand{\simge}{\gtrsim}
\newcommand{\simle}{\lesssim}
\newcommand{\A}{{\mathcal A}}
\newcommand{\B}{{\mathcal B}}
\newcommand{\al}{\alpha}
\newcommand{\la}{\lambda}
\newcommand{\ka}{\kappa}
\newcommand{\Om}{\Omega}
\newcommand{\eps}{\varepsilon}
\newcommand{\p}{{$p\mspace{1mu}$}}
\newcommand{\R}{\mathbf{R}}
\newcommand{\ub}{\bar{u}}
\newcommand{\ut}{\tilde{u}}

\newcommand{\vt}{\tilde{v}}
\newcommand{\ft}{\tilde{f}}
\newcommand{\fb}{\bar{f}}

\newcommand{\wt}{\widetilde{w}}
\newcommand{\Ft}{\widetilde{F}}
\newcommand{\Et}{\widetilde{E}}
\newcommand{\clG}{\itoverline{G}}
\newcommand{\clB}{\itoverline{B}}
\newcommand{\clBprime}{{\,\overline{\!B'}}}
\renewcommand{\phi}{\varphi}
\newcommand{\phit}{{\widetilde{\phi}}}
\newcommand{\phib}{{\itoverline{\phi}}}
\newcommand{\Wp}{W^{1,p}}
\newcommand{\Llp}{L^{1,p}}
\newcommand{\Hp}{H^{1,p}}
\newcommand{\Wploc}{W^{1,p}\loc}
\newcommand{\Hploc}{H^{1,p}\loc}
\newcommand{\eqv}{\ensuremath{\mathchoice{\quad \Longleftrightarrow \quad}{\Leftrightarrow}}
                {\Leftrightarrow}{\Leftrightarrow}}
\def\cprime{{\mathsurround0pt$'$}}
\numberwithin{equation}{section}
\newenvironment{ack}{\medskip{\it Acknowledgement.}}{}
\begin{document}

\authortitle{Jana Bj\"orn  and Abubakar Mwasa}
{Behaviour at infinity for solutions of a mixed boundary value problem via inversion}
\title{Behaviour at infinity for solutions of a mixed nonlinear elliptic boundary value problem via inversion}
\author{
Jana Bj\"orn \\
\it\small Department of Mathematics, Link\"oping University, 
\it\small SE-581 83 Link\"oping, Sweden \\ % \/{\rm ;} 
\it \small jana.bjorn@liu.se, ORCID\/\textup{:} 0000-0002-1238-6751
\\
\\
Abubakar Mwasa \\
\it\small Department of Mathematics, Link\"oping University, 
\it\small SE-581 83 Link\"oping, Sweden\/{\rm ;} \\
\it\small Department of Mathematics, Busitema University, 
\it\small P.O.Box 236, Tororo, Uganda    \\   %\/{\rm ;}
\it \small abubakar.mwasa@liu.se, amwasa@sci.busitema.ac.ug, ORCID\/\textup{:} 0000-0003-4077-3115
}

%\date{Preliminary version, \today}
\date{}

\maketitle

\noindent{\small
\begin{abstract}
\noi We study a mixed boundary value problem for the quasilinear
elliptic equation $\Div\A(x,\grad u(x))=0$ in an open infinite
circular half-cylinder 
with prescribed continuous Dirichlet data on a part of the boundary
and zero conormal derivative on the rest. 
We prove the existence and uniqueness of bounded 
weak solutions to the mixed problem
and characterize the regularity of the point at infinity in terms of
\p-capacities. 
For solutions with only Neumann data near the point at infinity we
show that they behave in exactly one of three possible ways, similar to the
alternatives in the Phragm\'en--Lindel\"of principle.
\end{abstract}
}

\bigskip

\noindent {\small \emph{Key words and phrases}: 
continuous Dirichlet data, existence and uniqueness of solutions,
mixed boundary value problem,  Phragm\'en--Lindel\"of trichotomy,
quasilinear elliptic equation, regularity at infinity, 
Wiener criterion.
}

\medskip

\noindent {\small Mathematics Subject Classification (2020):
Primary: 35J25.
Secondary:   35J62, 35B40.
}

\section{Introduction}
The \emph{Dirichlet problem} on a nonempty open set  $\Om\subset\R^n$
entails finding a function $u$ which solves a certain partial
differential equation in $\Om$ with the prescribed boundary data $u=f$
on $\bdy\Om$.  
When the boundary data $g:\bdy\Om\to\R$ are taken as the normal derivative
$\bdy u/\bdy \nu=g$, the problem is called a \emph{Neumann problem}.
More general directional derivatives, such as the conormal derivative 
$Nu$ considered below, are also possible.
In this paper, we study a \emph{mixed boundary value problem} 
for the quasilinear equation
\begin{equation}   \label{eq-divA}
\Div\A(x,\grad u(x))=0,
\end{equation} 
where a Dirichlet condition is prescribed on a part of the boundary,
while the rest carries the zero conormal derivative 
\[
Nu(x):= \A(x,\grad u(x))\cdot \nu(x) = 0,
\]
where $\nu(x)$ is the unit outer normal at $x$.
Equation \eqref{eq-divA} is considered in an infinite circular
half-cylinder and the vector-valued function $\A$ satisfies the
standard ellipticity conditions with a parameter $p>1$. 
The \p-Laplace equation $\Div(|\grad u|^{p-2}\grad u)=0$ is included
as a special case.

We prove the existence and uniqueness of bounded continuous weak solutions to the above
mixed boundary value problem with continuous Dirichlet boundary data
$f$ on a closed set $F\subset \bdy G$, which can be bounded or unbounded.
More precisely, if $f\in C(F)$ is such that (for unbounded $F$)
\begin{equation}  \label{eq-lim-at-infty-intro}
f(\infty):=\lim_{F\ni x\to\infty}f(x) \quad \text{exists and is finite,}
\end{equation}
then there exists a unique bounded continuous weak solution 
$u$ of \eqref{eq-divA} with zero conormal derivative 
on $\bdy G\setm F$ and such that 
\begin{equation*}	
\lim_{x\to x_0}u(x)=f(x_0) 
\quad \text{for all $x_0\in F$ outside a set of 
$C_p$-capacity zero,}
\end{equation*}
see Theorem~\ref{thm-uniq-sol} and Remark~\ref{rem-bdd-F0}.
Moreover,  $u$ is H\"older continuous at all points in the Neumann
boundary $\bdy G\setm F$.
For Dirichlet data of Sobolev type, existence and uniqueness 
in the Sobolev sense are proved in Theorem~\ref{thm-ex-Sob}.

We also characterize the regularity of the point at infinity by a
\emph{Wiener type criterion}.  
Roughly speaking, if the Dirichlet boundary $F$ is sufficiently thick
at $\infty$ in terms of a variational \p-capacity adapted to the
mixed problem, then for every $f\in C(F)$
satisfying~\eqref{eq-lim-at-infty-intro},
the unique solution of the above mixed problem satisfies 
\begin{equation}  \label{eq-reg-intro}
\lim_{x\to \infty}u(x)=f(\infty). 
\end{equation}
Conversely, thickness of $F$ at $\infty$ is also necessary in order for
\eqref{eq-reg-intro} to hold for all $f\in C(F)$, see Theorem~\ref{thm-Breg-infty}.
On the other hand, if $F$ is bounded, i.e.\ when only the Neumann
condition is used in a proximity of the point at infinity, then we show 
in Theorem~\ref{thm-cases} that each solution
$u$ of equation~\eqref{eq-divA} with zero conormal
derivative near $\infty$ behaves in one of the following three ways as $x\to\infty$:
\begin{enumerate}
\renewcommand{\theenumi}{\textup{(\roman{enumi})}}%
\renewcommand{\labelenumi}{\theenumi}
\item \label{intro-i} 
The solution has a finite limit $\disp u(\infty):=\lim_{x\to\infty} u(x)$.
\item \label{intro-ii} 
The solution tends roughly linearly to either $\infty$ or $-\infty$. 
\item \label{intro-iii} 
The solution changes sign and approaches both $\infty$ and $-\infty$, i.e.\ 
\[
\limsup_{x\to\infty} u(x)=\infty \quad \text{and} \quad
\liminf_{x\to\infty} u(x)=-\infty.
\]
\end{enumerate}
Similar trichotomy results at $\infty$ for solutions of the Neumann problem
for the linear uniformly elliptic equation $\Div(A(x)\grad u)=0$ 
were obtained in Ibragimov--Landis~\cite{IbrLan97},~\cite{IbrLan98},
Lakhturov~\cite{Lakh} and Landis--Panasenko~\cite{LanPanas}.
For (sub/super)so\-lu\-tions of various PDEs in unbounded domains with only Dirichlet data, 
similar alternative behaviour  is often referred to as Phragm\'en--Lindel\"of
principle and has been extensively studied.
See e.g.\ Gilbarg~\cite{Gilbarg}, Hopf~\cite{Hopf},
Horgan--Payne~\cite{HorgPay1},~\cite{HorgPay2},
Jin--Lancaster~\cite{JinLanc2000},~\cite{JinLanc2003}, 
Lindqvist~\cite{Lindqvist}, Lundstr\"om~\cite{Lundstr},
Quintanilla~\cite{Quintanilla}, Serrin~\cite{Serrin} and Vitolo~\cite{Vitolo}.

Compared with pure Dirichlet and Neumann problems, the literature on
mixed boundary value problems is less extensive, especially in unbounded domains.
Mixed problems are sometimes called \emph{Zaremba problems}, 
mainly in the Russian literature, 
since they were first considered for the Laplace equation $\Delta u=0$ 
by Zaremba~\cite{ZarProb} in 1910.
For linear uniformly elliptic equations of the type $\Div(A(x)\grad u)=0$, 
they were studied by e.g.\ Ibragimov~\cite{IbrDokl},~\cite{IbrSbor},
Kerimov~\cite{KerimovInfty},~\cite{KerimovWien} and
Novruzov~\cite{Novruzov}.
A~mixed problem for linear equations of nondivergence type was
considered in Cao--Ibragimov--Nazarov~\cite{CaoIbrNaz} and Ibragimov--Nazarov~\cite{IbrNaz}.

Existence of weak solutions for mixed and Neumann problems 
for linear operators in very general unbounded domains
was recently obtained using an exhaustion
with bounded domains by Chipot~\cite{Chipot} and Chipot--Zube~\cite{ChipZube}.
Wi\'sniewski~\cite{Wisn10},~\cite{Wisn16} studied the decay at
infinity of solutions to mixed problems with coefficients approaching the Laplace operator
in unbounded conical domains. 
On the other hand, nonexistence Liouville type results for mixed problems of the
form $-\Delta u =f(u)$ in various unbounded domains were obtained in 
Damascelli--Gladiali~\cite{DamGlad}.
A priori estimates, unique solubility and traces for linear elliptic systems of
divergence type in Ahlfors regular cylindrical domains with Dirichlet 
and Neumann data independent of the last variable were obtained in 
Auscher--Egert~\cite{AuschEg} by the so-called DB-approach adapted from the upper half-space.
All the above papers deal with linear operators.
We point out that even for $p=2$, equation~\eqref{eq-divA} considered
here can be nonlinear.
This happens for example when $\A(x,q):=a(q/|q|)q$, where 
$a$ is a sufficiently smooth strictly positive scalar function on the unit
sphere; see Example~\ref{ex-nonlin-p=2}.

In Kerimov--Maz\cprime ya--Novruzov~\cite{KMV},
the regularity of the point at infinity for the Zaremba problem for the
Laplace equation $\Delta u=0$  
in an infinite half-cylinder was characterized by means of a Wiener
type criterion.
A similar problem for certain weighted linear elliptic equations was
studied in Bj\"orn~\cite{JB} 
and more recently for the \p-Laplace equation in Bj\"orn--Mwasa~\cite{BM}.
The results in this paper partially extend the ones in~\cite{KMV},
\cite{JB} and \cite{BM} to general quasilinear elliptic equations 
of the form~\eqref{eq-divA}, but we also address other properties
of the solutions.

In order to achieve our results, we make use of the change of
variables introduced in Bj\"orn~\cite{JB}, and later adopted by
Bj\"orn--Mwasa~\cite{BM}, to transform the infinite 
half-cylinder $G$ and the quasilinear elliptic
equation~\eqref{eq-divA} into a unit half-ball and a degenerate
elliptic equation 
\begin{equation} \label{eq-B-xi}
\Div\B(\xi,\grad\ut(\xi))=0,
\end{equation} 
with the $A_p$-weight $w(\xi)=|\xi|^{p-n}$, see Section~\ref{sect-cyl-ball}.

After the above transformation, we are able to eliminate the Neumann
data by reflecting  
the unit half-ball and equation~\eqref{eq-B-xi} to the whole unit
ball, leaving only the Dirichlet data. 
This is done in Section~\ref{sect-pro-ope-B}.
In Section~\ref{sect-existence},
we then use tools for Dirichlet 
problems from Heinonen--Kilpel\"ai\-nen--Martio~\cite{HKM}
and Bj\"orn--Bj\"orn--Mwasa~\cite{BBM} to prove the existence and
uniqueness of continuous weak solutions to the mixed boundary value
problem for~\eqref{eq-divA} in the cylinder, with zero conormal
derivative and continuous Dirichlet data.

In Section~\ref{sect-Wiener}, we show that regularity at infinity for
the mixed problem for~\eqref{eq-divA} is equivalent to the regularity
of the origin for the Dirichlet problem for~\eqref{eq-B-xi}.
This can in turn be characterized in terms of weighted \p-capacities
using a Wiener criterion, provided by~\cite[Theorem~21.30]{HKM}
and Mikkonen~\cite{M}.
This Wiener criterion for~\eqref{eq-B-xi} is then transferred back to the 
cylinder to characterize the regularity at infinity for~\eqref{eq-divA}
by means of a variational \p-capacity adapted to the cylinder.

Finally, in Section~\ref{Sect-remov}, we use estimates 
from~\cite[Sections~6 and~7]{HKM} 
for capacitary potentials and singular 
solutions of~\eqref{eq-B-xi} to prove the above
Phragm\'en--Lindel\"of type trichotomy \ref{intro-i}--\ref{intro-iii}.

\begin{ack}
J.~B. was partially supported by the Swedish Research Council grant
2018-04106.
A.~M. was supported by 
the SIDA (Swedish International Development
Cooperation Agency)
project 316-2014  
``Capacity building in Mathematics and its applications'' 
under the SIDA bilateral program with the Makerere University 2015--2020,
contribution No.\ 51180060.
\end{ack}

\section{Notation and preliminaries} \label{sect-pre-not}

Let $G= B'\times (0,\infty)\subset \R^n$ be an open 
infinite circular half-cylinder, where 
\[
B'=\{x'\in\R^{n-1}:|x'|<1\}
\]
 is the unit ball in $\R^{n-1}$.
Points in $\R^n=\R^{n-1}\times \R$, $n\geq 2$, are denoted by
$x=(x',x_n)=(x_1,\cdots,x_{n-1}, x_n)$. 
Let $F$ be a closed subset of the closure $\clG$ of $G$. 
Assume that $F$ contains the base $B'\times \{0\}$ of $G$.
We consider the following mixed boundary value problem 
\begin{equation}  \label{mixed-bvp}
\begin{cases}
  \Div\A(x,\grad u)=0,    &\text{in } G\setm F,\\ 
  u=f,& \text{on } F_0:=F\cap\bdy(G\setm F),\text{ (Dirichlet data)}, \\
  Nu:= \A(x,\grad u)\cdot \nu =0
&\text{on } \bdy G\setm F, \text{ (generalized Neumann data)},
\end{cases}
\end{equation}
where $Nu$ is the conormal derivative, $\nu$ is the unit outer normal of $G$
and $\bdy G$ denotes the boundary of $G$.

Let $1<p<\infty$ be fixed. 
The mapping $\A: \clG \times \R^n\to \R^n$ in \eqref{mixed-bvp}
is assumed to satisfy the standard ellipticity and boundedness conditions:
\begin{itemize} 
\item $\A(\cdot,q)$ is measurable for all $q\in \R^n$,
\item $\A(x,\cdot)$ is continuous for a.e.\ $x\in \clG$,
\item There are constants $0<\al_1\le \al_2<\infty$ such that for all 
$q, q_1, q_2\in \R^n$, $0\ne\la\in\R$ and a.e.\ $x\in \clG$,
\begin{align}
\A(x,q) \cdot q &\ge \al_1 |q|^p, \label{eq-ell-A} \\
|\A(x,q)| &\le \al_2 |q|^{p-1}, \label{eq-bdd-A} \\
\A(x,\la q) & = \la |\la|^{p-2} \A(x,q),  \label{eq-homog-A}
\end{align}
and 
\begin{equation}
(\A(x,q_1)-\A(x,q_2))\cdot(q_1-q_2) >0 \quad\text{when }q_1\neq q_2.
\label{eq-monot-A} 
\end{equation}
\end{itemize}

The quasilinear elliptic equation  $\Div\A(x,\grad u)=0$
and the conormal derivative in~\eqref{mixed-bvp}
will be considered in the weak sense as follows.

\begin{defi}    \label{defi-weak-divA}
A function 
\[
u\in\Wp\loc(\clG \setm F) 
:= \{u|_{\clG\setm F}: u \in \Wp\loc(\R^n\setm F)\},
\]
is a \emph{weak solution} of the equation $\Div\A(x,\grad u)=0$
in $G\setm F$ with \emph{zero conormal derivative} 
on $\partial G\setm F$ if the integral identity 
\begin{equation} \label{eq-weak-sol-divA}
\int_{G\setm F} \A(x,\nabla u) \cdot \nabla\phi\, dx=0
\end{equation}
holds for all $\phi \in C_0^{\infty}(\R^n\setm F)$,
where $\cdot$ denotes the scalar product in $\R^n$ and
$C_0^{\infty}(\Omega)$ is the space of all infinitely many 
times continuously differentiable functions with compact support in $\Omega\subset\R^n$.
\end{defi}

\begin{remark}
If the whole boundary $\bdy G$ is contained in $F$, then there is no Neumann 
condition and the  mixed boundary value problem reduces to a purely Dirichlet problem.
\end{remark}

As usual, the local space $\Wp\loc(\Om)$ (and later $\Hp\loc(\Om,w)$) 
consists of those functions $u$
which belong to $\Wp(\Om')$ (resp.\ $\Hp(\Om',w)$) for all $\Om'\Subset\Om$,
where $\Om'\Subset\Om$ means that $\overline{\Om'}$ is a compact subset of $\Om$.

As mentioned in the introduction, equation~\eqref{eq-divA} can
  be nonlinear even for $p=2$, as illustrated by the following example.

\begin{example}   \label{ex-nonlin-p=2}
For $p=2$ and $x,q\in\R^n$, let 
\begin{equation}   \label{eq-def-A}
\A(x,q) = \begin{cases}  0 & \text{if } q=0, \\
           a\bigl( \frac{q}{|q|} \bigr) q & \text{if } q\ne0, 
        \end{cases}
\end{equation}
where 
the scalar function $a$ is strictly positive and continuous on the unit
sphere $\bdy B(0,1)$ in $\R^n$ and such that 
\begin{equation}   \label{eq-a/a}
\frac{a(\theta')}{a(\theta)} > \frac{1-\sin\al}{1+\sin\al}
\quad \text{for any $\theta,\theta'\in\bdy B(0,1)$ with 
    $\theta\cdot \theta'= \cos\al >0$,} 
\end{equation}
where $\al=\al(\theta,\theta')$ is the acute angle between
  $\theta$ and $\theta'$.
Then $\A$ satisfies the ellipticity
conditions~\eqref{eq-ell-A}--\eqref{eq-monot-A}.
 
Indeed, the only nontrivial verification is that of~\eqref{eq-monot-A}.
For this, we can clearly assume that $|q|=1$ and
\[
0<a(\theta')\le a(\theta), \quad \text{where } 
\theta=q  
\text{ and } \theta'=\frac{q'}{|q'|}.
\]
Condition~\eqref{eq-monot-A} then means that the angle between the
vectors $q'-q$ and $\frac{a(\theta')}{a(\theta)}q'-q$  is strictly less
than $\frac{\pi}{2}$.
Since $a(\theta')/a(\theta)\le1$, a geometric consideration shows that
this is clearly satisfied if $\al\ge\frac{\pi}{2}$ or if 
\[
\al<\frac{\pi}{2} \text{ and }  |q'|\le \frac{1}{\cos\al}.
\]
So assume that $\al<\frac{\pi}{2}$ and $|q'|>1/\cos\al$.
Aplying the cosine theorem to the triangles spanned by the vectors $q$
and $q'$, or by
\[
q \text{ and } q'':=\frac{a(\theta')}{a(\theta)}q',
\]
respectively, as well as to the triangle spanned by $q'-q$ and
$q''-q$, we see that the angle between $q'-q$ and $q''-q$ is
$<\frac{\pi}{2}$ if and only if
\[
|q|^2 + |q'|^2 -2|q|\,|q'| \cos\al + |q|^2 + |q''|^2 -2|q|\,|q''| \cos\al
> |q'-q''|^2.
\]
Since $|q|=1$, this is equivalent to 
\[
1 - \biggl( 1+\frac{a(\theta')}{a(\theta)} \biggr) |q'| \cos\al 
   + \frac {a(\theta')}{a(\theta)} |q'|^2 >0,
\]
i.e.
\[
\frac{a(\theta')}{a(\theta)} > \frac{|q'|\cos\al -1}{|q'|(|q'|-\cos\al)}.
\]
Maximizing the right-hand side over $|q'|> 1/\cos\al$, we find that
its maximum is attained when 
\[
|q'|=\frac{1+\sin\al}{\cos\al} \quad \text{and equals }
\frac{1-\sin\al}{1+\sin\al}.
\]
This justifies the requirement~\eqref{eq-a/a}.

A concrete example of a function satisfying~\eqref{eq-a/a} is 
$a(\theta)=e^{\theta\cdot q_0}$, for any fixed vector $q_0\in\R^n$
with $|q_0|<1/\sqrt2$.
Indeed, with this choice, 
\[
\frac{a(\theta')}{a(\theta)} = e^{(\theta'-\theta)\cdot q_0} 
> 1 - \frac{|\theta'-\theta|}{\sqrt2},
\]
while
\[
|\theta'-\theta|^2 = 2 - 2\theta\cdot \theta' = 2 - 2\cos\al
\le 2\sin^2 \al,
\]
from which~\eqref{eq-a/a} readily follows.

Clearly, a modification of~\eqref{eq-def-A} and~\eqref{eq-a/a}
can be made so that $\A(x,q)$ and $a(x,\theta)$ also
depend on  $x\in\R^n$.
\end{example}

Throughout the paper, unless otherwise stated, $C$ will denote any positive 
constant whose real value is not important and need not be the same 
at each point of use. 
It can even vary within a line. 
By $a\simle b$, we mean that there exists a positive constant $C$, independent 
of $a$ and $b$, such that $a\leq Cb$.
Similarly, $a\simge b$ means $b\simle a$, while $a\simeq b$ stands for
$a\simle b\simle a$.

\section{Transformation of the half-cylinder}  
\label{sect-cyl-ball}

In this section, the quasilinear elliptic operator $\Div\A(x,\grad u)$
in $G$ is shown to
correspond to a weighted quasilinear elliptic operator on the unit half-ball.
We will use the following change of variables 
introduced in Bj\"orn~\cite[Section~3]{JB}.

Let $\ka>0$ be a fixed constant.
Define the mapping 
\[
T:\R^n \longrightarrow T(\R^n) = \R^n\setm\{(\xi',\xi_n)\in \R^n: \xi'=0
\text{ and }\xi_n\leq 0\}
\]
by $T(x',x_n)=(\xi',\xi_n)$, where
\begin{equation}   \label{eq-def-xi-from-x}
\xi'= \frac{2e^{-\ka x_n} x'}{1+|x'|^2} \quad \text{and} \quad
\xi_n=\frac{e^{-\ka x_n}(1-|x'|^2)}{1+|x'|^2}.
\end{equation}
We will use $x=(x',x_n)$  for points in $\clG$ and 
$\xi=(\xi',\xi_n)=(\xi_1,\ldots,\xi_{n-1},\xi_n)\in \R^n$ for points 
in the space transformed by $T$.
Note that 
\begin{align*}
T(G) &= \{\xi\in\R^n: |\xi|<1\text{ and } \xi_n>0\}, \\ 
T(\clG) &= \{\xi\in\R^n: 0<|\xi|\le1\text{ and } \xi_n\ge0\}
\end{align*}
are the open and the closed upper unit half-ball, respectively, with the
origin $\xi=0$ removed. 
From \eqref{eq-def-xi-from-x} it is easy to see that
\[
|\xi|=|T(x)|=e^{-\ka x_n} \to 0 \quad\text{as } x_n \to \infty,
\]
that is, the point 
at infinity for the half-cylinder $G$ corresponds to 
the origin $\xi=0$ in $T(G)$.

The mapping $T$ is a smooth diffeomorphism between $\R^n$ and $T(\R^n)$, 
see Bj\"orn--Mwasa~\cite[Lemma~3.1]{BM}.
A direct calculation shows that the inverse mapping $T^{-1}$ of $T$
is given by
\begin{equation*}   
x'= \frac{\xi'}{|\xi|+\xi_n} \quad \text{and} \quad
x_n=-\frac{1}{\ka}\log|\xi|.
\end{equation*}

In the following lemma we show how the operator $\Div\A(x,\grad u)$
on the half-cylinder is transformed under $T$ to the unit half-ball,
cf.~\cite[Section~3]{BM}.

\begin{lem}              \label{lem-int-id-Ball}
Let $u,v\in \Wp(\Om)$ for some open set $\Om\subset G$ and let
$\ut= u\circ T^{-1}$ and $\vt= v\circ T^{-1}$.
Then for any measurable set $A\subset \Om$,
\begin{equation}  \label{eq-divA-to-B}
\int_{A} \A(x,\nabla u)\cdot\nabla v\, dx
=\int_{T(A)} \B(\xi,\grad\ut)\cdot\grad \vt \,d\xi,
\end{equation}
where $\B$ is for $\xi=Tx\in T(\clG)$ and $q\in\R^n$ defined by
\begin{equation}   \label{eq-def-B}
\B(\xi,q)= |J_T(x)|^{-1} dT(x)\A(x,dT^*(x)q). 
\end{equation}
\end{lem}

Here, $J_T(x) = \det(dT(x))$ denotes the Jacobian of $T$ at $x$
and $dT^*(x)$ is the transpose of the differential $dT(x)$ of $T$ at $x$,
seen as $(n\times n)$-matrices.
In~\eqref{eq-def-B}, both $q$ and $\A(x,dT^*(x)q)$ are regarded as column 
vectors for the matrix multiplication to make sense.

\begin{proof}
First, we rewrite the scalar product on the left-hand side of \eqref{eq-divA-to-B}
using matrix multiplication as
\[
\A(x, \grad u) \cdot\grad v 
     =(\grad v)^*\A(x,\grad u),
\]
where both $\A(x, \grad u)$ and $\grad v$ are seen as column vectors.
Using the change of variables $\xi=T(x)$, together with the chain rule 
\begin{equation*}  
\grad u(x)=dT^*(x)\grad \ut(\xi), 
\quad \text{where }\xi=T(x), 
\end{equation*}
we get
\begin{align*}
&\int_{A}(\grad v)^* \A(x,\grad u) \,dx \\
& \quad \quad \quad
= \int_{T(A)}(dT^*(x)\grad\vt)^*\A(x, dT^*(x)\grad\ut)|J_T(x)|^{-1} \,d\xi \\
&\quad \quad \quad
=\int_{T(A)}|J_T(x)|^{-1}dT(x)\A(x, dT^*(x)\grad\ut)\cdot\grad\vt \,d\xi\\
&\quad \quad \quad
=\int_{T(A)}\B(\xi,\grad\ut)\cdot\grad\vt\,d\xi. \qedhere
\end{align*}
\end{proof}

In view of the integral identity \eqref{eq-weak-sol-divA}, 
Lemma~\ref{lem-int-id-Ball} shows
that the quasilinear equation (\ref{eq-divA}) 
on $G\setm F$ will be transformed by $T$ into the equation
\begin{equation}   \label{eq-div-B}
\Div \B(\xi,\grad \ut)=0  \quad \text{on } T(G\setm F),
\end{equation}
with a proper interpretation of the function spaces and
the zero Neumann condition. 

To prove the fundamental properties of the transformed operator 
$\B(\xi,\grad\ut)$, we will need the following estimates from 
Bj\"orn--Mwasa~\cite[Lemma~3.3]{BM}.

\begin{lem}   \label{lem-JB}
There exist constants $C_1,C_2>0$ such that 
if $x,y\in B'\times\R$ and $x_n\le y_n$, then
\begin{equation*} 
C_1e^{-\ka y_n}|x-y|\le |T(x)-T(y)|\le C_2e^{-\ka x_n}|x-y|.
\end{equation*}
In particular, if $x\in B'\times\R$ 
and  $q\in\R^n$ then  
\[
|dT^*(x)q|\simeq |dT(x)q|\simeq e^{-\ka x_n}|q|
\quad \text{and}\quad 
|J_{T}(x)|\simeq e^{-\ka nx_n},
\]
where the comparison constants in $\simeq$ depend on $\ka$, but
are independent of $x$ and $q$.
 \end{lem}

\section{Properties of \texorpdfstring{$\Div \B(\xi,\grad \ut)$}{Div B(x,grad u)} and removing the 
\penalty-10000 Neumann data}
\label{sect-pro-ope-B}
In order to apply the theory of degenerate elliptic equations developed 
in Heinonen--Kilpel\"ainen--Martio~\cite{HKM}, we need to first show that 
the mapping $\B$ in~\eqref{eq-div-B} satisfies ellipticity 
assumptions similar to \eqref{eq-ell-A}--\eqref{eq-monot-A}.

\begin{thm} \label{thm-B-ellipt-bdd}
The mapping $\B:T(\clG)\times\R^n\to\R^n$, defined by \eqref{eq-def-B}, 
satisfies 
for all $q\in\R^n$ and a.e.\ $\xi\in T(\clG)$
the following ellipticity and boundedness conditions
\[
\B(\xi,q)\cdot q \simge \wt(\xi)|q|^p  \quad   \text{and}  \quad
|\B(\xi,q)| \simle \wt(\xi)|q|^{p-1},
\]
where $\wt(\xi)=|\xi|^{p-n}$ is a weight function and the comparison constants 
in $\simge$ and $\simle$ are independent of $\xi$ and $q$.
\end{thm}

\begin{proof}
From \eqref{eq-def-B}, we have that for all $q\in\R^n$ and a.e.\ 
$\xi=Tx\in T(\clG)$,
\begin{align*}
\B(\xi,q)\cdot q &=|J_T(x)|^{-1}dT(x)\A(x,dT^*(x)q)\cdot q\\
&=|J_T(x)|^{-1}\A(x,dT^*(x)q)\cdot (dT^*(x)q).
\end{align*}
Now applying \eqref{eq-ell-A} together with Lemma~\ref{lem-JB}, we get that
\[
\B(\xi,q)\cdot q \simge e^{\ka nx_n}|dT^*(x)q|^p\simeq e^{-\ka(p-n)x_n}|q|^p=\wt(\xi)|q|^p,
\]
which completes the proof of the first part. 
For the second part, we have using \eqref{eq-def-B}, \eqref{eq-bdd-A}
and Lemma~\ref{lem-JB} that 
\begin{align*}
|\B(\xi,q)|&=|J_T(x)|^{-1}|dT(x)\A(x,dT^*(x)q)|
\simeq |J_T(x)|^{-1}e^{-\ka x_n}|\A(x,dT^*(x)q)| \\
&\simle e^{(n-1)\ka x_n} |dT^*(x)q|^{p-1}
\simeq e^{-\ka(p-n)x_n}|q|^{p-1}= \wt(\xi)|q|^{p-1}.\qedhere
\end{align*}
\end{proof}

\begin{thm}\label{thm-monot}
The mapping $\B:T(\clG)\times\R^n\to\R^n$, defined by \eqref{eq-def-B}, 
satisfies for a.e.\  $\xi\in T(\clG)$ and all $q_1,q_2\in\R^n$ the monotonicity condition
\[
(\B(\xi,q_1)-\B(\xi,q_2))\cdot(q_1-q_2)>0 \quad\text{when }q_1\neq q_2.
\]
\end{thm}
\begin{proof}
Using \eqref{eq-def-B}, we have for all $q_1,q_2\in\R^n$,
\begin{align*}\label{eq-monot}
&(\B(\xi,q_1)-\B(\xi,q_2))\cdot(q_1-q_2)\\
&\qquad=|J_T(x)|^{-1}((dT(x)\A(x,dT^*(x)q_1)-dT(x)\A(x,dT^*(x)q_2))\cdot(q_1-q_2))\\
&\qquad=|J_T(x)|^{-1}((\A(x,dT^*(x)q_1)-\A(x,dT^*(x)q_2))\cdot(dT^*(x)q_1-dT^*(x)q_2)).
\end{align*}
Since $|J_T(x)|^{-1}>0$, we have by~\eqref{eq-monot-A} that the last expression 
is always nonnegative.
Moreover, it is zero if and only if $dT^*(x)q_1=dT^*(x)q_2$, implying that 
$q_1=q_2$ as $dT^*(x)$ is invertible.
\end{proof}

It is clear from Theorems~\ref{thm-B-ellipt-bdd} and~\ref{thm-monot},
together with the homogeneous condition~\eqref{eq-homog-A},  that the assumptions (3.4)--(3.7) 
in Heinonen--Kilpel\"ainen--Martio~\cite{HKM} are satisfied for $\B$ 
with the weight function 
$\wt(\xi)=|\xi|^{p-n}$, $\xi\in\R^n\setm\{0\}$. 
Moreover, $\B(\xi,q)$ is measurable in $\xi$ and continuous in $q$. 

The weight $\wt(\xi)$ belongs to the Muckenhoupt $A_p$ class and 
the associated measure $d\mu(\xi)=\wt(\xi)\,d\xi$ is doubling and 
supports a \p-Poincar\'e inequality on $\R^n$. 
Such weights are suitable for the study of partial differential equations 
and Sobolev spaces, see 
\cite[Chapters~15 and~20]{HKM} for a detailed exposition.
We follow \cite[Chapters~1 and~3]{HKM} giving the following definitions.

\begin{defi}
For an open set $\Om\subset\R^n$, the weighted Sobolev space 
$\Hp_0(\Om,\wt)$ is the completion of $C_0^\infty(\Om)$ with respect to the norm
\[
\|u\|_{\Hp(\Om,\wt)}:=\biggl(\int_{\Om} \bigl( |u(\xi)|^p
           +|\nabla u(\xi)|^p \bigr) \wt(\xi)\,d\xi \biggr)^{1/p}.
\]
Similarly, $H^{1,p}(\Om,\wt)$ is the completion of the set
\[
\{\phi\in C^{\infty}(\Om): \|\phi\|_{\Hp(\Om,\wt)} <\infty\}
\] 
in the $\Hp(\Om,\wt)$-norm.
\end{defi}

If the weight $\wt\equiv1$, then the symbol $\wt$ is dropped 
and in this case we have the usual Sobolev space $\Wp(\Om)$.

\begin{defi}
A function $u\in\Hploc(\Om,\wt)$ in an open set $\Om\subset\R^n$ is said to be 
a \emph{weak solution} of the equation $\Div\B(\xi,\grad u)=0$ 
if for all functions $\phi\in C_0^\infty(\Om)$, the following 
integral identity holds
	\begin{equation}    \label{int-B}
	\int_{\Om}\B(\xi,\grad u)\cdot\grad\phi\,d\xi=0.
	\end{equation}
	\end{defi}

To use the tools developed in Heinonen--Kilpel\"ainen--Martio~\cite{HKM} 
for Dirichlet problems, the part of the boundary, $T(\bdy G\setm F)$, 
where the zero conormal derivative is prescribed, will be removed.
This will be done by a reflection in the hyperplane 
\[
\{\xi\in\R^n:\xi_n=0\}.
\]
By a reflection, we mean the  mapping $P:\R^n\to\R^n$ defined by  
\[
P\xi=P(\xi',\xi_n)=(\xi',-\xi_n).
\]
Let $D$ be the open set consisting of $T(G\setm F)$ together with 
$PT(G\setm F)$ and $T(\bdy G\setm F)$, i.e.
\[
D=B(0,1)\setm \Ft, \quad\text{where } \Ft=T(F)\cup PT(F)\cup\{0\}.
\]
The point at infinity in $G$ corresponds to the origin $\xi=0$ in $B(0,1)$.
Note that $\Ft$ is closed and that the base $B'\times\{0\}$ of $G$ 
is mapped onto the upper unit half-sphere $\{\xi\in\bdy B(0,1):\xi_n>0\}$. 
By assumption, the base $B'\times\{0\}\subset F$ and so the whole boundary 
$\bdy D\subset\Ft$ carries the Dirichlet condition.

Extend $\B(\xi,q)$ from $T(\clG)$ to the reflected half-ball $PT(\clG)$ by  
\begin{equation} \label{eq-TG--PTG}
\B(\xi,q)= \begin{cases} 
P\B(P\xi,Pq) & \text{if }\xi_n<0,\\
0 & \text{if } \xi=0.\end{cases}
\end{equation}
The standard ellipticity and monotonicity assumptions for $\B$ still hold  
after the extension  from $T(\clG)$ to $PT(\clG)$. 

With this reflection, we will now be able to eliminate the Neumann boundary data 
on $T(\bdy G\setm F)$  so that only the Dirichlet data on $\bdy D$ remain.
In Theorem~\ref{thm-G-TG}, we will show that $u\in\Wploc(\clG\setm F)$ 
is a weak solution of the equation \eqref{eq-divA} in $G\setm F$ with 
zero conormal derivative on $\bdy G\setm F$ if and only if 
the symmetric reflection of $u\circ T^{-1}$
is a weak solution of \eqref{eq-div-B} in $D$.
We shall use the function spaces identified in Bj\"orn--Mwasa~\cite{BM}.

\begin{lem}         \label{lem-intG-T}
{\rm (\cite[Lemma~4.6]{BM})}
Assume that $u\in L^1\loc(U)$ with the distributional
gradient $\grad u \in L^1\loc(U)$ for some open set 
$U\subset B'(0,R)\times\R$ and let $\ut=u\circ T^{-1}$.  
Then for any measurable set $A\subset U$,
\begin{align*}   
\int_{A}|\grad u|^p\,dx 
  &\simeq \int_{T(A)}|\grad \ut|^p \wt \,d\xi, \\
\int_{A}|u|^p e^{-p\ka x_n}\,dx
  &\simeq \int_{T(A)}|\ut|^p \wt 
\,d\xi, \nonumber 
\end{align*}
with comparison constants depending on $R$ but independent of $A$ and $u$.
\end{lem}

\begin{prop}  \label{prop-Wploc-Hploc}
Let $u\in \Wploc(\clG\setm F)$. Then the function
\begin{equation}  		\label{eq-def--ut}
\ut(\xi',\xi_n)= \begin{cases}
        (u\circ T^{-1})(\xi',\xi_n)    &\text{if } \xi\in T(\clG\setm F),\\ 
        (u\circ T^{-1})(\xi',-\xi_n) &\text{if } \xi\in PT(G\setm F),
        \end{cases}
\end{equation}
belongs to $\Hploc(D,\wt)$.
Conversely, if $\vt\in\Hploc(D,\wt)$, then $\vt\circ T\in\Wploc(\clG\setm F)$.
\end{prop}

\begin{proof}
By the definition of $\Wploc(\clG\setm F)$, we can assume that $u\in\Wploc(\R^n\setm F)$.
Then $u\in\Wp(U)$ for every $U\Subset\R^n\setm F$.
To show that $\ut\in\Hp\loc(D,\wt)$, let $B\Subset D$ be a ball. 
Assume without loss of generality that $B$ is centred in $T(\clG\setm F)$ and set $U:=T^{-1}(B)$.  
Choose a sequence $u_j\in C^\infty(\R^n\setm F)$ such that $u_j\to u$ in $\Wp(U)$.
In particular, $u_j$ are Lipschitz on $U$.
Define $\ut_j:=u_j\circ T^{-1}$ restricted to $B\cap T(\clG)$.

By Lemma~\ref{lem-JB}, the functions $\ut_j$ are Lipschitz.
Their extensions across the hyperplane $\xi_n=0$ to the whole $B$, 
by the reflection $\ut_j(\xi)=\ut_j(P\xi)$,  are still Lipschitz.
Lemma~\ref{lem-intG-T} implies that $\ut$ can be approximated 
in the $\Hp(B,\wt)$-norm by these Lipschitz extensions, that is
\[
\|\ut_j-\ut\|_{\Hp(B,\wt)}\simle \|u_j-u\|_{\Wp(U)}\to0\quad \text{as }j\to\infty.
\]
Hence, $\ut\in\Hp(B, \wt)$. 
Since $B$ was arbitrary, we have that $\ut\in\Hploc(D,\wt)$.

Conversely, let $V\Subset\R^n\setm F$. Then $T(V)\Subset D$ and hence $\vt\in\Hp(T(V),\wt)$.
By Lemma~\ref{lem-intG-T} and the fact that $e^{-p\ka x_n }\simeq1$ on $V$, we have that 
$\vt\circ T$ belongs to $\Wp(V)$.
Since $V$ was arbitrary, we get that
$\vt\circ T\in\Wploc(\clG\setm F)$. 
\end{proof}

\begin{thm}				\label{thm-G-TG}
Let $u\in\Wploc(\clG\setm F)$ be a weak solution of the equation 
\begin{equation}   \label{eq-DivA-G-F}
\Div\A(x,\grad u)=0
\quad \text{in $G\setm F$}
\end{equation}
 with zero conormal derivative on 
$\bdy G\setm F$, that is, the integral identity \eqref{eq-weak-sol-divA} 
holds for all $\phi\in C_0^\infty(\R^n\setm F)$.
Let $\ut$ be as in \eqref{eq-def--ut}. 
Then $\ut\in\Hploc(D,\wt)$ is a weak solution of the equation
\begin{equation}   \label{eq-DivB}
\Div\B(\xi,\grad \ut)=0\quad\text{in }D.
\end{equation}
Conversely, if $\ub\in\Hploc(D,\wt)$ is a weak solution of~\eqref{eq-DivB} 
such that $\ub=\ub\circ P$, then $\ub\circ T\in\Wploc(\clG\setm F)$,
with $\ub$ restricted to $T(\clG\setm F)$,
is a weak solution of the equation~\eqref{eq-DivA-G-F}  
with zero conormal derivative on $\bdy G\setm F$.
\end{thm}

\begin{proof} 
The fact that $\ut\in\Hploc(D,\wt)$ follows from Proposition~\ref{prop-Wploc-Hploc}.
To prove the first implication, let $\phib\in C_0^\infty(D)$ be an arbitrary 
test function.
Clearly, $\phib\circ T|_{\clG\setm F}$ is a restriction of a function from
$C_0^\infty(\R^n\setm F)$.
In Lemma~\ref{lem-int-id-Ball}, replace $v,\vt$ and $A$ with 
$\phib\circ T, \phib$ and $G\setm F$, respectively.
Lemma~\ref{lem-int-id-Ball} 
and the integral identity \eqref{eq-weak-sol-divA} then give 
\begin{equation}  		\label{eq-int-TG=G=0}
\int_{T(G\setm F)} \B(\xi,\grad\ut)\cdot\grad \phib \,d\xi 
=\int_{G\setm F}\A(x,\grad u)\cdot\grad(\phib\circ T)\,dx= 0.
\end{equation}
Now the change of variables $\zeta=P\xi$, together with \eqref{eq-TG--PTG} 
and the fact that $\ut=\ut\circ P$, 
yields
\begin{equation}  		\label{eq-int-PTG=TG}
\int_{PT(G\setm F)} \B(\xi,\grad\ut)\cdot\grad\phib \,d\xi 
= \int_{T(G\setm F)} \B(\xi,\grad \ut)\cdot\grad (\phib\circ P) \,d\xi,
\end{equation}
cf.\ \cite[Lemma~6.1]{BM}.
Since $\phib\circ P \in C_0^\infty(D)$, we see as 
in \eqref{eq-int-TG=G=0} that the last integral is zero.
Adding the left-hand sides of \eqref{eq-int-TG=G=0} and \eqref{eq-int-PTG=TG} 
shows that $\ut$ is a weak solution of \eqref{eq-DivB}.
	
Conversely, first we recall from \cite[Lemma~3.11]{HKM}
that if $\ub$ is a weak solution of \eqref{eq-DivB}, 
then the integral identity \eqref{int-B} holds for all test functions 
in $\Hp_0(D,\wt)$ with compact support in $D$.
Let $\phi\in C_0^\infty(\R^n\setm F)$ be an arbitrary test function.
Define 
\begin{equation*} 
\phit(\xi',\xi_n):= \begin{cases}
(\phi\circ T^{-1})(\xi',\xi_n)    &\text{if } \xi\in T(\clG),\\ 
(\phi\circ T^{-1})(\xi',-\xi_n) &\text{if } \xi\in PT(G),\\
0 &\text{otherwise.} \end{cases}
\end{equation*}
Note that $\phit=\phit\circ P$ and that $\phit$ 
has compact support in $D$.
Proposition~\ref{prop-Wploc-Hploc} therefore implies that 
$\phit\in\Hp_0(D,\wt)$.
Since $\ub=\ub\circ P$, the integral identity~\eqref{int-B}, 
tested with $\phit$, gives
\begin{equation*} 
\int_{T(G\setm F)} \B(\xi,\grad\ub)\cdot\grad \phit \,d\xi
    +\int_{PT(G\setm F)} \B(\xi,\grad\ub)\cdot\grad \phit \,d\xi
=\int_D\B(\xi,\grad\ub)\cdot\grad\phit\,d\xi=0. 
\end{equation*}
	
Observe that the integrals on the left-hand side are the same as 
in \eqref{eq-int-PTG=TG} with $\ut,\phib$ replaced by $\ub,\phit$ 
and are thus equal.
It follows that
\[
\int_{T(G\setm F)} \B(\xi,\grad\ub)\cdot\grad \phit \,d\xi=0.
\]

Replacing $u,v,\vt$ and $A$ in Lemma~\ref{lem-int-id-Ball} by 
$\ub\circ T,\phi,\phit$ and $G\setm F$, respectively, we then get
\[
\int_{G\setm F} \A(x,\grad(\ub\circ T))\cdot\grad \phi \,dx
=\int_{T(G\setm F)} \B(\xi,\grad\ub)\cdot\grad \phit \,d\xi=0.
\]
Since $\phi$ was chosen arbitrarily,  we conclude that $\ub\circ T$ is 
a weak solution of \eqref{eq-DivA-G-F} with zero conormal derivative,
as in Definition~\ref{defi-weak-divA}.
Lastly, Proposition~\ref{prop-Wploc-Hploc} shows that 
$\ub\circ T\in\Wploc(\clG\setm F)$.
\end{proof}

\begin{remark}  \label{rem-cont-sol}
The weak solution $\ut$ can be modified on a set of measure zero, so
that it becomes H\"older continuous  in $D$, 
see Heinonen--Kilpel\"ainen--Martio~\cite[Theorems~3.70 and~6.6]{HKM}. 
Hence, for the corresponding continuous representative of $u$,
the limit 
\[
\lim_{G\setm F\in x\to x_0}u(x)
\]
exists and is finite  
for every $x_0$ on the Neumann boundary $\bdy G\setm F$.
Moreover, $u$ is H\"older continuous at $x_0$, by Lemma~\ref{lem-JB}.
\end{remark}

\section{Existence and uniqueness of  the solutions}
\label{sect-existence}

In this section, we shall prove the existence and uniqueness of weak solutions 
to the equation $\Div\A(x,\grad u)=0$ in $G\setm F$ with zero conormal derivative on $\bdy G\setm F$
and continuous Dirichlet data.
We shall also show that the solution attains its continuous Dirichlet data
except possibly on a set of Sobolev $C_p$-capacity zero.
Recall that $F$ is a closed subset of $\clG$, which contains the base 
$B'\times\{0\}$, and that the Dirichlet boundary data $f$ are prescribed 
on $F_0:=F\cap\bdy(G\setm F)$.

\begin{defi}
Let $K\subset\R^n$ be a compact set. The Sobolev $(p,\wt)$-capacity of $K$ is 
\begin{equation*}
C_{p,\wt}(K)=\inf_v\int_{\R^n}(|v|^p+|\grad v|^p)\wt\,dx,
\end{equation*}
where the infimum is taken over all $v\in C_0^\infty(\R^n)$ 
(or equivalently, all continuous $v\in\Hp_0(\R^n,\wt)$)
such that $v\ge 1$ on $K$, see 
Heinonen--Kilpel\"ainen--Martio~\cite[Section~2.35 and~Lemma~2.36]{HKM}.
\end{defi}

The capacity $C_{p,\wt}$ can be extended to general sets as a
Choquet capacity, see \cite[Chapter~2]{HKM}. 
In particular, for all Borel sets $E\subset\R^n$,
\begin{equation}	\label{eq-Cp-choq}
C_{p,\wt}(E)=\sup\{C_{p,\wt}(K):K\subset E\text{ compact}\}.
\end{equation}
We also say that a property holds \emph{$C_{p,\wt}$-quasieverywhere}
if the set where it fails has zero $C_{p,\wt}$-capacity.
If $\wt\equiv1$, then we have the usual Sobolev $C_p$-capacity.
The following lemma shows that $T$ preserves sets of zero capacity. 

\begin{lem} \label{lem-Cpw=Cp=0}
Let $E\subset \clG$.  
Then $C_p(E)=0$ if and only if $C_{p,\wt}(T(E))=0$.
\end{lem} 

\begin{proof}
The fact that if $C_{p,\wt}(T(E))=0$ then $C_p(E)=0$ follows from Lemma~6.6 
in Bj\"orn--Mwasa~\cite{BM}.

Conversely, assume that $C_p(E)=0$.
Replacing $E$ by a Borel set $E_0\supset E$ with zero capacity,
we can assume that $E$  is a Borel set.
Let $K\subset T(E)$ be compact.  
Then $T^{-1}(K)\subset E$ is compact and hence $C_p(T^{-1}(K))=0$.
For $\eps>0$, choose $\phi\in C^\infty_0(\R^n)$ such that $\phi\ge 1$ 
on $T^{-1}(K)$ and $\|\phi\|_{\Wp(\R^n)}<\eps$.
Multiplying $\phi$ by a suitable cut-off function,
we can assume that $\phi(x)=0$ when $x_n\le-1$.
Define 
\begin{equation*}  	
\phit(\xi',\xi_n):= \begin{cases}
(\phi\circ T^{-1})(\xi',\xi_n)    &\text{if } \xi\in T(\clBprime\times\R),\\ 
(\phi\circ T^{-1})(\xi',-\xi_n) &\text{if } \xi\in PT(B'\times\R),\\
0 &\text{if }\xi=0. \end{cases}
\end{equation*}
Then $\phit$ is Lipschitz continuous, by Lemma~\ref{lem-JB},
belongs to $\Hp_0(\R^n,\wt)$ and is such that $\phit\ge1$ on $K$.
Using Lemma~\ref{lem-intG-T} and the fact that $\phi(x)=0$ when
$x_n\le-1$, we have 
\[
\|\phit\|_{\Hp(\R^n,\wt)}  \simeq \|\phit\|_{\Hp(T(B'\times\R),\wt)}
\simle \|\phi\|_{\Wp(B'\times\R)}\le\|\phi\|_{\Wp(\R^n)}<\eps,
\]
with comparison constants independent of $\phi$ and $\eps$.
Letting $\eps\to 0$, gives that $C_{p,\wt}(K)=0$ and hence \eqref{eq-Cp-choq} concludes the proof.
\end{proof}

The following existence and uniqueness theorem is the main
  result of this section.
Note that there may also exist unbounded solutions.

\begin{thm}   \label{thm-uniq-sol}
Assume that $F_0$ is unbounded and let $f\in C(F_0)$ be such that 
\[
f(\infty):=\lim_{F_0\ni x\to\infty}f(x) \quad \text{exists and is finite}.
\]
Then there exists a unique bounded continuous weak solution 
$u\in\Wploc(\clG\setm F)$ of the mixed problem for the equation
$\Div\A(x,\grad u)=0$ in $G\setm F$ with zero conormal derivative 
on $\bdy G\setm F$ and such that 
\begin{equation}		\label{eq-limu=f}
\lim_{G\setm F\ni x\to x_0}u(x)=f(x_0) \quad \text{for $C_{p}$-quasievery
$x_0\in F_0$.}
\end{equation}
\end{thm}

\begin{remark}  \label{rem-bdd-F0}
The proof below shows that the conclusion of
Theorem~\ref{thm-uniq-sol} holds also when $F_0$ is bounded.
Moreover, since the origin has $(p,\wt)$-capacity zero, the solution
$u$ is then independent of the value $f(\infty)$ assigned to $\xi=0$, 
\end{remark}

\begin{proof}
Define
\begin{equation}  		\label{eq-def--ft}
\ft(\xi):= \begin{cases}
        (f\circ T^{-1})(\xi)    &\text{if } \xi\in T(F_0),\\
        (f\circ (PT)^{-1})(\xi) &\text{if } \xi\in PT(F_0\cap G),\\
        f(\infty) &\text{if } \xi=0.\end{cases}
\end{equation}
Then $\ft\in C(\bdy D)$.
By Bj\"orn--Bj\"orn--Mwasa~\cite[Theorem~3.12]{BBM}, there exists 
a unique bounded continuous weak solution $\ub\in\Hp\loc(D,\wt)$
of the equation \eqref{eq-DivB} such that
\begin{equation}	\label{eq-limub=ft}
\lim_{D\ni\xi\to\xi_0}\ub(\xi)=\ft(\xi_0) 
\quad \text{for $C_{p,\wt}$-quasievery $\xi_0\in \bdy D$,}
\end{equation}
i.e.\ for all $\xi_0\in \bdy D\setm Z$ for some 
set $Z\subset\bdy D$  with $C_{p,\wt}(Z)=0$.
Note that \eqref{eq-TG--PTG} holds in $D$ and $\ft=\ft\circ P$.
So \eqref{eq-limub=ft} gives
\[
\lim_{D\ni\xi\to\xi_0}\ub(P\xi)=\lim_{D\ni\xi\to P\xi_0}\ub(\xi)=\ft(P\xi_0)=\ft(\xi_0)
\]
for all $\xi_0\in \bdy D\setm P(Z)$.
That is, $\ub\circ P$ also satisfies~\eqref{eq-limub=ft} and 
$\ub\circ P\in\Hploc(D,\wt)$.
Since $\phi\circ P\in C_0^\infty(D)$ if and only if $\phi\in C_0^\infty(D)$, 
the change of variables $\zeta=P\xi$ together with~\eqref{eq-TG--PTG} 
shows that the integral identity \eqref{int-B} holds for $\ub\circ P$ as well.
Thus $\ub\circ P$ is also a bounded continuous weak solution of~\eqref{eq-DivB} 
in $D$, satisfying~\eqref{eq-limub=ft}.
By the uniqueness in \cite[Theorem~3.12]{BBM}, we conclude that
$\ub=\ub\circ P$.

Define $u:=\ub\circ T$, with $\ub$ restricted to $T(\clG\setm F)$.
Theorem~\ref{thm-G-TG} shows that $u$ is a continuous weak solution 
of $\Div\A(x,\grad u)=0$ in $G\setm F$ with zero conormal derivative,
as in Definition~\ref{defi-weak-divA}.
Since $\ub$ satisfies \eqref{eq-limub=ft}, it then follows that $u$
satisfies \eqref{eq-limu=f} for every $x_0\in F_0\setm Z_0$, where
$Z_0:=T^{-1}(Z\cap T(\clG))$ with $C_p(Z_0)=0$ by
Lemma~\ref{lem-Cpw=Cp=0}.  

To prove the uniqueness, suppose that $v\in\Wploc(\clG\setm F)$ 
is a bounded 
continuous weak solution of the equation $\Div\A(x,\grad u)=0$
satisfying \eqref{eq-limu=f} for 
all $x_0\in F_0\setm Z'_0$ with $C_p(Z'_0)=0$.
Let $\vt$ be as in \eqref{eq-def--ut} with $u$ replaced by $v$.
Then by Theorem~\ref{thm-G-TG}, $\vt$ is a continuous weak solution of~\eqref{eq-DivB}.

Since $v$ satisfies \eqref{eq-limu=f}, it follows that  $\vt$ satisfies 
\eqref{eq-limub=ft} for 
each $\xi_0\in\bdy D\setm Z'$, where $Z':=T(Z'_0)\cup PT(Z'_0)\cup\{0\}$. 
Now we have that $C_{p,\wt}(T(Z'_0))=0$ by Lemma~\ref{lem-Cpw=Cp=0},
and so $C_{p,\wt}(PT(Z'_0))=0$, by reflection.
The origin $0$ has zero $(p,\wt)$-capacity by \cite[Lemma~7.6]{BM}, and 
it follows by subadditivity that $C_{p,\wt}(Z')=0$. 

By \cite[Theorem~3.12]{BBM}, the solution of \eqref{eq-DivB} 
satisfying~\eqref{eq-limub=ft}  is unique, in other words $\vt=\ub$ and so $v=u$.
\end{proof}

The following definition is adopted from Bj\"orn--Mwasa~\cite{BM}.

\begin{defi}   \label{def-L-space-ka}
The space $\Llp_{\ka}(G\setm F)$ consists of all measurable functions
$v$ on $G\setm F$ such that the norm 
\[ 
\|v\|_{\Llp_\ka(G\setm F)} 
= \biggl( \int_{G\setm F} \bigl( |v(x)|^p e^{-p\ka x_n}
     +|\grad v(x)|^p \bigr) \,dx \biggr)^{1/p} < \infty,
\] 
where $\nabla v=(\partial_1v,\cdots,\partial_nv)$ is the 
distributional gradient of $v$.
The space $\Llp_{\ka,0}(\clG\setm F)$ is the completion of 
$C_0^{\infty}(\R^n\setm F)$ in the above $\Llp_\ka(G\setm F)$-norm.
\end{defi}

Note that the space $\Llp_\ka(G\setm F)$ is contained in 
$\Wploc(\clG\setm F)$.
The following result generalizes~\cite[Theorem~6.3]{BM}
to elliptic divergence type equations.

\begin{thm}   \label{thm-ex-Sob}
Let $f\in \Llp_\ka(G\setm F)$.
Then there exists a unique continuous weak solution $u\in \Llp_\ka(G\setm F)$ 
of the equation $\Div\A(x,\grad u)=0$ in $G\setm F$ with zero 
conormal derivative on $\bdy G\setm F$ and
such that \(u-f\in\Llp_{\ka,0}(\clG\setm F)\).
\end{thm}

\begin{proof}
Let $\ft$ be defined as in \eqref{eq-def--ut}, with $u$ replaced by $f$.
Then $\ft\in\Hp(D,\wt)$, by \cite[Proposition~5.3]{BM}.
By \cite[Theorems~3.17 and~3.70]{HKM}, there exists a unique continuous 
weak solution $\ut\in\Hp(D,\wt)$ of 
the degenerate equation~\eqref{eq-DivB} such that 
\(\ut-\ft\in\Hp_0(D,\wt)\).
Since $\B(\xi,Pq)=P\B(P\xi,q)$ by \eqref{eq-TG--PTG}, we infer from
\cite[Corollary~6.2]{BM} (with $\A(\xi,q)$ in~\cite{BM}
replaced by $\B(\xi,q)$) that $\ut=\ut\circ P$.
By Theorem~\ref{thm-G-TG}, $u:=\ut\circ T$ (with $\ut$ restricted to 
$T(G\setm F)$) is a weak solution of $\Div\A(x,\grad u)=0$ in $G\setm F$ 
with zero conormal derivative on $\bdy G\setm F$. 
Moreover, $u\in\Llp_\ka(G\setm F)$ by \cite[Proposition~5.3]{BM} and 
$u-f\in\Llp_{\ka,0}(\clG\setm F)$ by \cite[Proposition~5.5]{BM}.

To prove the uniqueness, suppose that 
$v\in\Llp_{\ka}(G\setm F)$ is a continuous weak solution of the equation 
$\Div\A(x,\grad u)=0$ in $G\setm F$ with zero conormal derivative
on $\bdy G\setm F$ 
and \(v-f\in\Llp_{\ka,0}(\clG\setm F)\). 
Let $\vt$ be as in \eqref{eq-def--ut}, with $u$ replaced by~$v$.
Theorem~\ref{thm-G-TG} then implies that $\vt$ satisfies 
the equation~\eqref{eq-DivB}.
Moreover, \cite[Proposition~5.5]{BM} shows that $\vt-\ft\in\Hp_0(D, \wt)$. 
From the uniqueness of solutions to \eqref{eq-DivB} we thus get that
$\vt=\ut$, and so $v=u$.
\end{proof}

\section{Boundary regularity at infinity}
\label{sect-Wiener}

We saw in Remark~\ref{rem-cont-sol} that weak solutions of the equation 
$\Div\A(x,\grad u)=0$ with zero conormal derivative are
continuous in $G\setm F$ and at the Neumann boundary $\bdy G\setm F$. 
Moreover, if the Dirichlet boundary data $f$ are continuous on $F_0$, 
then the solution is continuous at $F_0$, except possibly for a set of 
Sobolev $C_p$-capacity zero.
We now study continuity at the point at infinity.

We follow Bj\"orn--Mwasa~\cite[Section~8]{BM} giving the following definition.

\begin{defi}   \label{def-reg-infty}
Assume that $F$ is unbounded.
We say that the point at $\infty$ is \emph{regular} for 
the mixed problem~\eqref{mixed-bvp}
for the equation $\Div\A(x,\grad u)=0$ in $G\setm F$
with zero conormal derivative on $\bdy G\setm F$ 
if for all Dirichlet boundary data $f\in C(F_0)$ with a finite limit
\begin{equation}  \label{eq--f(infty)}
\lim_{F_0\ni x\to\infty} f(x) =:f(\infty),
\end{equation}
the unique bounded continuous weak solution $u$ of~\eqref{mixed-bvp}, provided by 
Theorem~\ref{thm-uniq-sol}, satisfies
\begin{equation}   \label{eq-reg--infty}
\lim_{G\setm F\ni x\to \infty} u(x) =f(\infty).
\end{equation}
\end{defi}

Remark~\ref{rem-bdd-F0} shows that the point at infinity is always
irregular when $F_0$ is bounded.
Note that due to the conormal derivative condition, the regularity at
$\infty$ in Definition~\ref{def-reg-infty} differs from the usual notion
of boundary regularity for the Dirichlet problem in unbounded domains,
as in e.g.\  
Heinonen--Kilpel\"ainen--Martio~\cite[Section~9.5]{HKM}.
The following is our first step in characterizing the regularity of the point
at infinity for the mixed problem~\eqref{mixed-bvp}.

\begin{prop}		\label{prop-reg-0-infty}
The point at $\infty$ is regular for the mixed problem~\eqref{mixed-bvp} 
for the equation $\Div\A(x,\grad u)=0$ in $G\setm F$ if and only if 
the origin $0\in\bdry D$ is regular with respect to the equation
\begin{equation}   \label{eq-divB-inD}
\Div\B(\xi,\grad\ut(\xi))=0 \quad \text{in } D,
\end{equation}
where $\B$ is as in \eqref{eq-def-B}.
\end{prop}

Before the proof of Proposition~\ref{prop-reg-0-infty}, 
we give the definition below, see \cite[Section~9.5]{HKM}.

\begin{defi}
A point $\xi_0\in\bdy D$ is regular for the equation~\eqref{eq-divB-inD} if
\begin{equation*} 
\lim_{D\ni\xi\to\xi_0}\ub(\xi)=\fb(\xi_0)\quad\text{for all }\fb\in C(\bdy D),
\end{equation*}
where $\ub$ is the Perron solution of \eqref{eq-divB-inD} 
with the boundary data $\fb$, as in \cite[Section~9.1]{HKM}.
\end{defi}

Note that $\fb$, being continuous, is resolutive, i.e.\ the upper and 
the lower Perron solutions coincide and are equal to the 
Perron solution $\ub$, see \cite[Theorem~9.25]{HKM}.
Moreover, by Bj\"orn--Bj\"orn--Mwasa~\cite[Theorem~3.12]{BBM}, 
the Perron solution $\ub$ is the only bounded continuous weak solution 
of~\eqref{eq-divB-inD} that attains the boundary values $\fb$ 
$C_{p,\wt}$-quasieverywhere on $\bdy D$ in the sense of
\eqref{eq-limub=ft}.

\begin{proof}[Proof of Proposition~\ref{prop-reg-0-infty}]
Assume that $0\in\bdy D$ is a regular point with respect 
to \eqref{eq-divB-inD}.
Let $f\in C(F_0)$ be such that the limit in \eqref{eq--f(infty)} exists and 
is finite.
Let $u$ be the unique bounded continuous weak solution of $\Div\A(x,\grad u)=0$ 
in $G\setm F$ with zero conormal derivative, 
provided for $f$ by Theorem~\ref{thm-uniq-sol}.
Define $\ut$ and $\ft$ as in \eqref{eq-def--ut} and \eqref{eq-def--ft}, 
respectively.
Then by Theorem~\ref{thm-G-TG}$, \ut$ is a bounded continuous 
weak solution of $\Div\B(\xi,\grad\ut)=0$ in $D$ and 
attains the boundary values $\ft$ $C_{p,\wt}$-quasieverywhere on  $\bdy D$.
Note that $\ft\in C(\bdy D)$. 
Thus, by \cite[Theorem~3.12]{BBM}, $\ut$ is the Perron solution of~\eqref{eq-divB-inD} with the boundary data $\ft$.

Since $0\in\bdy D$ is regular,  we have that  
\[
\lim_{G\setm F\ni x\to \infty} u(x) =\lim_{D\ni\xi\to0}\ut(\xi)=\ft(0)=f(\infty).
\]
Thus, $u$ satisfies \eqref{eq-reg--infty} and so $\infty$ is regular 
for the mixed problem for the equation $\Div\A(x,\grad u)=0$ in $G\setm F$.

Conversely, assume that $\infty$ is regular for the mixed 
problem~\eqref{mixed-bvp} and let $\fb\in C(\bdy D)$.
The function $\fb$ is not necessarily symmetric, so we consider 
\[
\fb_1=\min\{\fb,\fb\circ P\} \quad \text{and} \quad
\fb_2=\max\{\fb,\fb\circ P\},
\]
which are symmetric, i.e.\ $\fb_j=\fb_j\circ P$, $j=1,2$.
Let $\ub$ and $\ub_j$ be the Perron solutions of \eqref{eq-divB-inD} with
boundary data $\fb$ and $\fb_j$, respectively.
Define $u_j:=\ub_j\circ T$, $j=1,2$, with $\ub_j$ restricted to $T(\clG\setm F)$.
Then by Theorem~\ref{thm-G-TG}, $u_j$ are bounded continuous weak solutions of 
the mixed problem~\eqref{mixed-bvp} satisfying \eqref{eq-limu=f} with
boundary data $f_j:=\fb_j\circ T|_{F_0}$.
Note that $f_j$ are continuous on $F_0$ and that the limits
\[
\lim_{F_0\ni x\to\infty}f_j(x)=f_j(\infty) \quad \text{exist and are finite}.
\]
Since $\infty$ is regular for~\eqref{mixed-bvp}, the solutions $u_j$ satisfy
\[
\lim_{G\setm F\ni x\to\infty}u_j(x)=f_j(\infty) =\fb(0).
\]
It now follows that
\[
\lim_{D\ni \xi\to0}\ub_j(\xi)=\fb(0).
\] 
But $\fb_1\le\fb\le\fb_2$ and thus $\ub_1\le\ub\le\ub_2$ by the definition
of Perron solutions.
We conclude that
\[
\lim_{D\ni \xi\to0}\ub(\xi)=\fb(0),
\]
and so $0\in\bdy D$ is regular for the equation \eqref{eq-divB-inD}.
\end{proof}

Regular points with respect to $\Div\B(\xi,\grad\ut(\xi))=0$ in $D$ 
are characterized by the \emph{Wiener criterion}, 
see Heinonen--Kilpel\"ainen--Martio~\cite[Theorem~21.30\,(i)$\eqv$(v)]{HKM}  
and Mikkonen~\cite{M}.
Note that the definitions for regularity of a point $\xi\in\bdy D$ 
with respect to \eqref{eq-divB-inD} in terms of both Sobolev 
and Perron solutions are equivalent, see \cite[Theorem~9.20]{HKM}.

Recall that $D=B_1\setm\Ft$, where $\Ft=T(F)\cup PT(F)\cup\{0\}$. 
Note that 
the ball $B_r:=\{\xi\in\R^n:|\xi|<r\}$ in $B_1$ corresponds to
the truncated cylinder
\begin{equation*} 
G_t:=\{x\in\clG:x_n>t\}=\clBprime\times(t,\infty)\quad\text{with } t=-\frac{1}{\ka}\log r\ge0
\end{equation*}
in  $\clG$ and that $G_t$ contains the lateral boundary but not its base $B'\times\{t\}$.
With the above notation, we see that $G_{2t}$ and $G_{t-1}$ correspond
to $B_{r^2}$ and $B_{2r}$, respectively.
We follow \cite[Section~7]{BM} giving the following definition.

\begin{defi}  
Let $K\subset G_{t-1}$ be a compact set, where $t\ge1$.
The \emph{{\rm(}Neumann\/{\rm)} variational \p-capacity} of $K$ with respect to $G_{t-1}$ is 
\begin{equation*}
\cp_{p,G_{t-1}}(K)= \inf_v \int_{G_{t-1}}|\grad v|^p\,dx,	
\end{equation*} 
where the infimum is taken over all functions 
$v\in C_0^\infty(\R^n)$ satisfying $v\ge1$ on $K$  and $v=0$ on $\clG\setm G_{t-1}$.
\end{defi}

Just as $C_{p,\wt}$, the capacity $\cp_{p,G_{t-1}}$ is also a Choquet capacity.
This was proved in \cite[Section~7]{BM}, even though we only need 
$\cp_{p,G_{t-1}}$ for compact sets here.
In \cite[Section~8]{BM}, the Wiener criterion from \cite[Section~6.16]{HKM} is rewritten in terms of $\cp_{p,G_{t-1}}$ and we thus get the following criterion for the boundary regularity at $\infty$ for the mixed boundary value problem \eqref{mixed-bvp}.

\begin{thm} \label{thm-Breg-infty}
The point at $\infty$ is regular for the mixed boundary value problem 
\eqref{mixed-bvp}   
in $G\setm F$ if and only if the following condition holds
\begin{equation} 		\label{eq-B-infty}
\int_1^\infty\cp_{p,G_{t-1}}(F\cap(\clG_t\setm G_{2t}))^{1/(p-1)}\,dt=\infty.
\end{equation}
\end{thm}

The proof follows from \cite[Section~8]{BM},
with Proposition~\ref{prop-reg-0-infty} playing the role of Lemma~8.2 in \cite{BM}.
See Examples~8.7 and~8.8 in \cite{BM} for concrete sets satisfying or failing
the Wiener condition~\eqref{eq-B-infty}.

\section{General behaviour of the solutions at \texorpdfstring{$\infty$}{oo}}
\label{Sect-remov}

Our aim in this section is to show a Phragm\'en--Lindel\"of type 
trichotomy for the solutions of the equation $\Div\A(x,\grad u)=0$
in $G$ with zero conormal derivative $C_p$-quasieverywhere on $\bdy G\cap G_t$.
We start by showing that sets of Sobolev $C_p$-capacity zero 
are removable for the solutions.
Recall that $F$ is a closed subset of $\clG$.

For compact subsets of $G$, the following removability result is just 
\cite[Theorem~7.36]{HKM}.
Using the transformation $T$, we can remove also parts of the lateral
Dirichlet boundary and change them into the Neumann boundary.
This will make it possible to study the behaviour of the solutions 
at $\infty$.

Recall that $G_t:=\{x\in\clG:x_n>t\}$
is the truncated cylinder containing its lateral boundary
but not its base $B'\times\{t\}$.

 \begin{lem} \label{lem-remov}
Assume that for some $t\ge0$, the set $E:=F\cap G_t$ satisfies $C_p(E)=0$.
 Let $u\in\Wploc(\clG\setm F)$ be a continuous weak solution of 
$\Div\A(x,\grad u)=0$ in $G\setm F$ 
with zero conormal derivative on $\bdy G\setm F$. 
Assume that $u$ is bounded in the set $(G_t\setm G_\tau) \setm F$ 
for every $\tau>t$.

Then $u$ can be extended to $E$ as a continuous weak solution in 
$(G\setm F)\cup E$ with zero conormal derivative on 
$\bdy G\setm (F\setm G_t)$.
\end{lem}

\begin{proof}
Set $\Et:=T(E)\cup PT(E)$.
By Lemma~\ref{lem-Cpw=Cp=0}, the set $T(E)$ has Sobolev $(p,\wt)$-capacity zero.
By reflection, we have $C_{p,\wt}(\Et)=0$.
As in~\eqref{eq-def--ut}, define
\begin{equation}  \label{eq-def-ut}
\ut(\xi',\xi_n)= \begin{cases}
        (u\circ T^{-1})(\xi',\xi_n)    &\text{if } \xi\in T(\clG\setm F),\\ 
        (u\circ T^{-1})(\xi',-\xi_n) &\text{if } \xi\in PT(G\setm F).
        \end{cases}
 \end{equation}
Then by Theorem~\ref{thm-G-TG}, $\ut$ is a continuous weak solution of 
$\Div\B(\xi, \grad\ut)=0$ in $D$, which is bounded in $D\cap (B_r\setm B_\rho)$ 
for every $\rho>0$, where $r=e^{-\ka t}$.
Note that $\Et\setm B_\rho$ is relatively closed in $B_r\setm B_\rho$.

Since $C_{p,\wt}(\Et)=0$, we have by 
Heinonen--Kilpel\"ainen--Martio~\cite[Theorem~7.36]{HKM} that $\ut$ can be
extended to $\Et$ so that it is a continuous weak solution in $B_r\setm B_\rho$
for every $\rho>0$, and thus also in $D\cup (B_r\setm \{0\})$. 
The desired extension of $u$ is then given by $\ut\circ T$. 
\end{proof}

Removability and behaviour of the solutions at $\infty$ are addressed 
in the rest of the section. 
The following two lemmas provide suitable lower and upper bounds for the
solutions.

\begin{lem}	\label{lem-max-Gt}
Assume that for some $t\ge0$, the set $E:=F\cap G_t$ is empty.
Let $u\in\Wploc(\clG\setm F)$ be a weak continuous solution of
$\Div\A(x,\grad u)=0$ in $G\setm F$ with zero conormal derivative on
$\bdy G\setm F$. 
Then $u$ is bounded in the set $G_{t'}\setm G_\tau$  for every $\tau>t'>t$.

If, moreover, $u(x)\le0$ when $x_n=t$,
then  either $u\le0$ in $G_t$ or there exist $A>0$ and $\tau_0>t$ such that for all $\tau\ge \tau_0$,
\begin{equation*} 
\max_{x_n=\tau}u(x)\ge A(\tau -t).
\end{equation*}
\end{lem}

\begin{proof}
Define $\ut$ as in \eqref{eq-def-ut}.
Then by Theorem~\ref{thm-G-TG}, $\ut\in\Hploc(D,\wt)$ is a continuous 
weak solution of $\Div\B(\xi,\grad\ut(\xi))=0$ in $D$.
In particular, since $E$ is empty, $\ut\in\Hploc(B_r\setm\{0\},\wt)$ 
is a continuous weak solution in $B_r\setm\{0\}$, where $r=e^{-\ka t}$.
This immediately implies the boundedness of $u$ in $G_{t'}\setm G_\tau$ 
for every $\tau>t'>t$.

Next, if $u(x)\le0$ when $x_n=t$, then
$\ut(\xi_0)\le 0$ for all $\xi_0\in\bdy B_r$. 
Hence, by \cite[Theorem~7.40]{HKM}, we have that either $\ut\le0$ 
in $B_r\setm\{0\}$ or
\begin{equation} \label{Ineq-cap-Br}
\liminf_{\rho\to0}(\cp_{p,\wt}(B_{\rho},B_r))^{1/(p-1)}\max_{\bdy B_{\rho}}\ut>0,
\end{equation}
where $\cp_{p,\wt}$ is the variational capacity associated with the
weight $\wt$, as in \cite[Chapter~2]{HKM}.
Inequality \eqref{Ineq-cap-Br} reveals that there exist constants
$a>0$ and $r_0>0$ such that for all $0<\rho\le r_0$, we have 
\begin{equation}  \label{Ineq-cap-m}
\max_{\bdy B_{\rho}}\ut\ge a(\cp_{p,\wt}(B_{\rho},B_r))^{1/(1-p)}.
\end{equation}
By Bj\"orn--Mwasa~\cite[Lemma~7.6]{BM}, we have for all $\rho<r$,
\[
(\cp_{p,\wt}(B_{\rho},B_r))^{1/(1-p)} \simge \log\frac{r}{\rho}
=\ka(\tau-t),\quad\text{where }\tau=-\frac{1}{\ka}\log\rho>t
\]
and the constant in $\simge$ depends only on $n$ and $p$.
Substituting in \eqref{Ineq-cap-m} and using the definition~\eqref{eq-def-ut} 
of $\ut$, concludes the proof. 
\end{proof}

\begin{lem}  \label{lem-est-with-pot}
Assume that for some $t\ge0$, the set $E:=F\cap G_t$ is empty.
Let $u\in\Wploc(\clG\setm F)$ be a weak continuous solution of $\Div\A(x,\grad u)=0$ 
in $G\setm F$ with zero conormal derivative on $\bdy G\setm F$. 
Assume that $u\ge 0$ in $G_t$.  
Then  there exists $A_0\ge0$ such that for all $\tau\ge t+\tfrac{1}{\ka}\log2$,
\[
\max_{x_n=\tau}u(x) \le A_0(\tau-t). 
\]
\end{lem}

\begin{proof}
Define $\ut$ as in~\eqref{eq-def-ut}.
As in the proof of Lemma~\ref{lem-max-Gt},
$\ut$ is a continuous weak solution of $\Div\B(\xi,\grad\ut(\xi))=0$
in $B_r\setm\{0\}$, where $r=e^{-\ka t}$.
For $0<\rho<\frac{1}{2} r$, let $\vt$ be the potential of $\clB_\rho$ in $B_r$,
i.e.\ the continuous weak solution of $\Div\B(\xi,\grad\vt)=0$ in 
$B_r\setm \clB_\rho$ 
with boundary values $1$ on $\clB_\rho$ and $0$ on $\bdy B_r$.  
Then by Heinonen--Kilpel\"ainen--Martio~\cite[Lemma~6.21]{HKM}, 
there exists $c>0$ such that
\[
\vt(\xi)\ge c\biggl(\frac{\cp_{p,\wt}(B_\rho,B_r)}{\cp_{p,\wt}(B_{r/2},B_r)}\biggr)^{1/(p-1)}
\quad\text{for all }\xi\in \clB_{r/2}.
\]
From \cite[Lemma~7.6]{BM} and \cite[Theorem~2.18]{HKM}, we thus have for all 
$\xi\in \clB_{r/2}$ that 
\begin{equation} \label{eq-vt-potent}
\vt(\xi)\simge \cp_{p,\wt}(B_\rho,B_r)^{1/(p-1)}\simge
\Bigl(\log \frac{r}{\rho}\Bigr)^{-1}=\frac{1}{\ka(\tau-t)},
\end{equation}
where $\tau=-\frac{1}{\ka}\log\rho>t$ and the constants in $\simge$ depend only on $c$, $n$ and $p$.

Let $m_\rho$ be the minimum  of $\ut$ on the sphere $\bdy B_\rho$.
Using the boundary values of $\vt$ and $\ut$ on both $\bdy B_\rho$ and $\bdy B_r$, 
it follows from the comparison principle \cite[Lemma~3.18]{HKM} 
and from \eqref{eq-vt-potent} that
\begin{equation*}  
\ut\ge m_\rho\vt \simge \frac{m_\rho}{\tau-t}  
\quad\text{in }  \clB_{r/2}\setm\clB_\rho.
\end{equation*}
The Harnack inequality for $\ut$ on the sphere $\bdy B_\rho$ then gives for any 
fixed  $\xi_0\in\bdy B_{r/2}$ that
\[
\max_{x_n=\tau} u(x) =
\max_{\bdy B_\rho}\ut(x) \simle m_\rho \simle \ut(\xi_0)(\tau-t),
\]
where the constants in $\simle$ depend only on $\A$, $\ka$, $n$ and $p$.
\end{proof}

We are now ready to conclude the paper with the following trichotomy
for the solutions of the mixed problem at $\infty$, when $F$ is
negligible near $\infty$.

\begin{thm} \label{thm-cases}
Assume that for some $t\ge0$, the set $E:=F\cap G_t$ is such that $C_p(E)=0$.
Let $u\in\Wploc(\clG\setm F)$ be a weak continuous solution of $\Div\A(x,\grad u)=0$ in $G\setm F$ with zero conormal derivative on $\bdy G\setm F$.
Assume that $u$ is bounded in the set $(G_t\setm G_\tau) \setm F$
for every $\tau>t$.

Then there exist constants $\tau_0>t$, $M$, $M_0$ and $A_0\ge A\ge0$,
such that exactly one of the following holds.
\begin{enumerate}
\renewcommand{\theenumi}{\textup{(\roman{enumi})}}%
\renewcommand{\labelenumi}{\theenumi}
\item \label{it-i} The solution $u$ is bounded in $G_{\tau_0}$ and
the limit\/
\[
\lim_{G\setm F\ni x\to\infty} u(x)=:u(\infty)
\] 
exists and is finite. 
Moreover, for some $\al\in(0,1]$ and all $x\in G_{\tau_0}$,
\begin{equation} \label{eq-u-holder}
|u(x)-u(\infty)|\simle e^{-\ka\al x_n}.
\end{equation}
\item \label{it-ii} 
The solution $u$ tends roughly linearly either to $\infty$ or to
  $-\infty$, i.e.\ either
\begin{equation}   \label{eq-u-lin-to-infty}
M+A\tau \le u(x',\tau) \le M_0+A_0\tau 
\quad \text{for all $x'\in B'$ and $\tau\ge\tau_0$,}
\end{equation}
or \eqref{eq-u-lin-to-infty} holds for $-u$ in place of $u$.
\item \label{it-iv} 
The solution changes sign and approaches both $\infty$ and $-\infty$. 
More precisely, 
\[
\disp\max_{x_n=\tau}u(x)\ge M+A\tau \text{ and } 
\disp\min_{x_n=\tau}u(x)\le M_0-A\tau \quad
\text{for all $\tau\ge \tau_0$.}
\]
\end{enumerate}
\end{thm}

\begin{proof}
By Lemma~\ref{lem-remov}, we can assume that $E$ is empty.
Define $\ut$ as in~\eqref{eq-def-ut}.
Then by Theorem~\ref{thm-G-TG}, the function $\ut$
is a continuous weak solution of $\Div\B(\xi, \grad\ut)=0$ in $B_r\setm\{0\}$,
where $r=e^{-\ka t}$. 
Let $r'\le\tfrac12r$. 
By the continuity of the solution $\ut$ in $B_r\setm\{0\}$,
there exist constants $m',M'$ such that  $m'\le \ut(\xi)\le M'$ 
for all $\xi\in\bdy B_{r'}$.
Hence, it follows from the definition of $\ut$ that $m'\le u(x)\le M'$
when $x_n=t':=-\frac{1}{\ka}\log r'>t$. 

Applying the second part of Lemma~\ref{lem-max-Gt} to $G_{t'}$ 
with $u$ replaced by $u-M'$ and $m'-u$,
respectively, shows that there exist $A'>0$ and $\tau_0>t'$ such that the following
two statements hold:
\begin{enumerate}
\renewcommand{\theenumi}{\textup{(\alph{enumi})}}%
\renewcommand{\labelenumi}{\theenumi}
\item \label{it-M'} 
$u\le M'$ in $G_{t'}$ or $\disp\max_{x_n=\tau}u(x)\ge M'+A'(\tau-t')$ for all $\tau\ge\tau_0$,
\item \label{it-m'} 
$u\ge m'$ in $G_{t'}$ or $\disp\min_{x_n=\tau}u(x)\le m'-A'(\tau-t')$ for all $\tau\ge\tau_0$.
\end{enumerate}
Combining the first two alternatives in \ref{it-M'} and \ref{it-m'} gives \ref{it-i},
while the second alternatives give \ref{it-iv}.
Since $C_{p,\wt}(\{0\})=0$, we can in the bounded case \ref{it-i}
use \cite[Theorem~7.36]{HKM}
and extend $\ut$ to $0$ so that it becomes a continuous weak solution 
of $\Div\B(\xi, \grad\ut)=0$ in $B_r$.
By \cite[Theorem~6.6]{HKM}, $\ut$ is H\"older continuous at the origin, 
which shows that \eqref{eq-u-holder} holds.

The remaining alternatives will lead to case~\ref{it-ii}.
If $u\ge m'$ in $G_{t'}$ and 
\[
\max_{x_n=\tau}u(x)\ge M'+A'(\tau-t')
\quad \text{for all $\tau\ge\tau_0$,}
\]
then using the Harnack inequality 
for $\ut-m'$ on the sphere $\bdy B_\rho$ with $\rho=e^{-\ka\tau}$, we see that for some
$C>0$ independent of~$\tau\ge\tau_0$,
\[
\min_{x_n=\tau}(u-m') \ge C \max_{x_n=\tau}(u-m') \ge C(M'-m'+A'(\tau-t')).
\]
This proves the lower bound in \eqref{eq-u-lin-to-infty},
while the upper bound follows 
from Lemma~\ref{lem-est-with-pot} applied to $u-m'$.
The second case of \ref{it-ii}, i.e.\ \eqref{eq-u-lin-to-infty}
  for $-u$, follows in a similar way by combining 
$u\le M'$ with the second alternative in \ref{it-m'}.
\end{proof}

We finish by giving a concrete example illustrating the cases in 
Theorem~\ref{thm-cases}.

\begin{example}
Let $G=(-1,1)\times(0,\infty)\subset\R^2$ and $F=[-1,1]\times\{0\}$.
The linear function $u(x_1,x_2)=ax_2+b$, where $a,b\in\R$,
satisfies $\Delta u=0$ in $G$ and $\bdy u/\bdy\nu=0$ on the lateral
boundary $\bdy G\setm F$.
The cases \ref{it-i} and \ref{it-ii}
follow if $a=0$, $a>0$ and $a<0$, respectively. 
Also, consider the function 
\[
u(x_1,x_2)= e^{\frac{\pi}{2} x_2}\sin\frac{\pi}{2} x_1,
\]
which is easily verified to satisfy $\Delta u=0$ in $G$.
Then \ref{it-iv} is achieved when $\sin\tfrac\pi2 x_1$ attains its  maximum
and minimum at $[-1,1]$, respectively.
\end{example}

\end{document}